\documentclass[a4paper,11pt]{article}

\usepackage[T1]{fontenc}
\usepackage{lmodern}

\usepackage{dsfont}

\usepackage[latin1]{inputenc}
\usepackage{amsmath}
\usepackage{amsthm}
\usepackage{amssymb}
\usepackage{mathrsfs}
\usepackage{graphicx}
\usepackage[all]{xy}
\usepackage{hyperref}

\usepackage{stmaryrd}
\usepackage{caption}

\usepackage{abstract} 

\newtheorem{thm}{Theorem}[section]
\newtheorem{cor}[thm]{Corollary}
\newtheorem{claim}[thm]{Claim}
\newtheorem{fact}[thm]{Fact}

\newtheorem{lemma}[thm]{Lemma}
\newtheorem{prop}[thm]{Proposition}

\theoremstyle{definition}
\newtheorem{definition}[thm]{Definition}

\newtheorem{remark}[thm]{Remark}
\newtheorem{question}[thm]{Question}

\newtheorem{conj}[thm]{Conjecture}

\def\rquotient#1#2{%
	\makeatletter
	\raise.3ex\hbox{$#1$}/\lower.3ex\hbox{$#2$}%
	\makeatother
}	

\makeatletter
\newcommand{\subjclass}[2][2010]{%
	\let\@oldtitle\@title%
	\gdef\@title{\@oldtitle\footnotetext{#1 \emph{Mathematics subject classification.} #2}}%
}
\newcommand{\keywords}[1]{%
	\let\@@oldtitle\@title%
	\gdef\@title{\@@oldtitle\footnotetext{\emph{Key words and phrases.} #1.}}%
}
\makeatother

\newcommand{\Address}{{
		\bigskip
		\small
		
		\textsc{Institut Montpellierain Alexander Grothendieck, 499-554 Rue du Truel, 34090 Montpellier, France.}\par\nopagebreak
		\textit{E-mail address}: \texttt{anthony.genevois@umontpellier.fr}
\medskip

		\textsc{Institut de Math\'ematiques de Jussieu-Paris Rive Gauche, Place Aur\'elie Nemours, 75013 Paris, France.}\par\nopagebreak
		\textit{E-mail address}: \texttt{romain.tessera@imj-prg.fr}
\medskip
		
}}

\title{Measure-scaling quasi-isometries}
\date{\today}
\author{Anthony Genevois and Romain Tessera}

\subjclass{Primary 20F65. Secondary 20F69.}
\keywords{Wreath products, lamplighter groups, quasi-isometric classification}

\begin{document}

\maketitle

\begin{abstract} A measure-scaling quasi-isometry between two connected graphs is a quasi-isometry that is quasi-$\kappa$-to-one in a natural sense for some $\kappa>0$. For non-amenable graphs, all quasi-isometries  are quasi-$\kappa$-to-one for any $\kappa>0$, while for amenable ones there exists at most one possible such $\kappa$. For an amenable graph $X$, we show that the set of possible $\kappa$ forms a subgroup of $\mathbb{R}_{>0}$ that we call the (measure-)scaling group of $X$. This group is invariant under measure-scaling quasi-isometries. In the context of Cayley graphs,  this implies for instance that two uniform lattices in a given locally compact group have same scaling groups.
We compute the scaling group in a number of cases. For instance it is all of $\mathbb{R}_{>0}$ for lattices in Carnot groups, SOL or solvable Baumslag Solitar groups, but is a (strict) subgroup $\mathbb{Q}_{>0}$ for lamplighter groups over finitely presented amenable groups. 
\end{abstract}

\tableofcontents

\section{Introduction}

Famously, quasi-isometries between non-amenable graphs of bounded degree all lie at finite distance from bijections \cite{MR1693847, MR1700742}, while this may not be the case for amemable graphs \cite{MR1616135, MR1616159} or even for finitely generated groups \cite{MR2730576, MR3318424}. In this article, we study a subclass of quasi-isometries for which the default of lying at finite distance from a bijection can be quantified. More precisely, we say that a quasi-isometry $f : X \to Y$ is \emph{measure-scaling} if there exists some $\kappa>0$, referred to as the \emph{(measure-)scaling factor}, such that $f$ is \emph{quasi-$\kappa$-to-one}, i.e. there exists a constant $C >0$ such that
$$\left| \kappa|A| - |f^{-1}(A)| \right| \leq C \cdot | \partial A| \text{ for all $A \subset Y$ finite.}$$
Motivations for such a definition can be found below, but it is worth noticing that an $n$-to-one quasi-isometry is automatically quasi-$n$-to-one; that a quasi-isometry lies at bounded distance from a bijection if and only if it is quasi-one-to-one \cite{MR1700742}; and that, when $\kappa$ is rational, we have the following characterization (see Proposition \ref{prop:kappa} for a more complete result):

\begin{prop}\label{prop:rational}
Let $X$ and $Y$ be bounded degree graphs. 
If $\kappa=m/n$ is rational, then any quasi-$\kappa$-to-one quasi-isometry from $X$ to $Y$ lies at bounded distance from the composition of a surjective $m$-to-one map with the canonical injection $X\to X\times \mathbb{Z}/n\mathbb{Z}$
\end{prop}

Observe that, for non-amenable graphs, all quasi-isometries are measure-scaling quasi-isometries for all possible scaling factors, in particular $\kappa=1$. At the opposite, one easily checks that, for amenable graphs, a quasi-isometry can only be quasi-$\kappa$-to-one for one value of $\kappa$ (see Lemma \ref{rem:scalingNonAmen} below). Also, we emphasize that not all quasi-isometries are measure-scaling quasi-isometries. For instance, the map $f:$ $\mathbb{Z} \to \mathbb{Z}$ defined as $f(n)=n$ for $n\geq 0$ and $f(n)=2n$ for $n<0$ is a quasi-isometry but not a measure-scaling quasi-isometry.
Nevertheless, we will see below that ``classical'' quasi-isometries between finitely generated groups are measure-scaling: for instance any two uniform lattices in a locally compact group are measure-scaling quasi-isometric.

The main question addressed in this article is the following, mainly motivated by our work \cite{QIwreath} (see Theorem~\ref{thm:QIlamp} below):

\begin{question}\label{questionIntro}
Given a finitely generated (amenable) group $G$, for which $\kappa \in \mathbb{R}_+$ does there exist a quasi-$\kappa$-to-one quasi-isometry $G \to G$?
\end{question}
In \cite[Section~4]{MR2139686}, Tullia Dymarz asks ``which amenable groups admit $n$-to-one quasi-isometries?'' By Proposition \ref{prop:rational}, her question can be seen as a particular case of the latter. 
 Our first observation is that all the possible values for $\kappa$ define a multiplicative subgroup in $\mathbb{R}$ (see Proposition \ref{prop:EasyKappa} for a more detailed statement).

\begin{prop}\label{propIntro:scQI}
For amenable graphs with bounded degree, measure-scaling quasi-isometries are stable under composition, the scaling factor behaving multiplicatively. 
\end{prop}

Given a metric space $X$, we recall that 
\[\mathrm{QI}(X):= \{ \text{quasi-isometry } X \to X \} / \text{bounded distance}\]  forms a group called the \emph{quasi-isometry} group of $X$. We deduce from Proposition \ref{propIntro:scQI} that measure-scaling quasi-isometries of an amenable graph $X$ (up to finite distance) form a subgroup of the quasi-isometry group. Let us denote it by $\mathrm{QI_{sc}}(X)$.  Moreover, the map that associates to every measure-scaling quasi-isometry the scaling factor induces a morphism $\mathrm{QI_{sc}}(X)\to \mathbb{R}_{>0}$.
 The range of this morphism, the \emph{(measure-)scaling group} of $X$ denoted by $\mathrm{Sc}(X)$, provides us with an algebraic measure-scaling quasi-isometry invariant of our graph.

When studying Question~\ref{questionIntro}, it turns out to be natural to extend the definition of measure-scaling quasi-isometries from graphs with bounded degree to metric measure spaces. More precisely, a measurable map $f : (X,\mu) \to (Y,\nu)$ between two standard metric measure spaces is \emph{measure-scaling} if there exists some $\kappa>0$, referred to as the \emph{(measure-)scaling factor}, such that $f$ is \emph{quasi-$\kappa$-to-one}, i.e.
$$\left| \kappa \cdot \nu(A) - \mu(f^{-1}(A)) \right| \leq C \cdot \nu (\partial_2 A) \text{ for all thick subspace $A \subset Y$.}$$
We refer to Section \ref{sec:MMspaces} for all the needed definitions. The main justification of such a generalization comes from the following observation, which constitutes the main contribution of our article: if a given group $G$ can be realized as a uniform lattice in some compactly generated locally compact group $H$, then $G$ and $H$ share the same scaling group but the scaling group of $H$ may be easier to compute. More precisely (see Theorem~\ref{thm:FundThm} for a more general result),

\begin{prop}\label{prop:scalingcopsi}
Let $G$ and $H$ be compactly generated locally compact groups equipped with left Haar measures $\mu_G$ and $\mu_H$ and word metrics associated to compact generating subsets. A continuous proper cocompact morphism $\phi:G\to H$ is a measure-scaling quasi-isometry. In particular,
$\mathrm{Sc}(H)= \mathrm{Sc}(G)$.
\end{prop}

As a particular case, if $H$ is a uniform lattice in a locally compact group $G$, then the inclusion map is a measure-scaling quasi-isometry. Combining Corollary \ref{cor:LatticesScaing} with \cite{MR1700742}, we obtain the following interesting fact.
 
 \begin{prop}\label{prop:latticesBilpii}
Two uniform lattices with same covolume in a compactly generated locally compact group are bi-Lipschitz equivalent.  
\end{prop}

Observe that an automorphism of a locally compact compactly generated group is a quasi-isometry. Since it multiplies the Haar measure by a given factor $\lambda>0$, it is a measure-scaling quasi-isometry of scaling factor $\kappa=\lambda^{-1}$. The set of scaling factors of automorphisms is therefore a subgroup of $\mathrm{Sc}(G)$, which we denote by $\mathrm{Sc_{Aut}}(G)$. The group  $\mathrm{Sc_{Aut}}(G)$ is obviously trivial if $G$ is discrete. But it may happen that $\mathrm{Sc}(G)$ coincides with $\mathrm{Sc_{Aut}}(\bar{G})$ for some locally compact group $\bar{G}$ containing $G$ as a lattice.
As an easy illustration, notice that $\mathbb{Z}^n$ is uniform lattice in $\mathbb{R}^n$. Since $\mathrm{Sc_{Aut}}(\mathbb{R}^n)=\mathbb{R}_{>0}$,
we deduce that $\mathrm{Sc}(\mathbb{Z}^n)=\mathrm{Sc}(\mathbb{R}^n) = \mathbb{R}_{>0}$. 
This remark allows us to compute the scaling group in a variety of examples, including solvable Baumslag-Solitar groups and lattices in Carnot groups or $\mathrm{SOL}$ (see Corollary \ref{cor:scalinggroupR}). 
Lamplighter groups provide examples  of groups for which $\mathrm{Sc}(G)\neq \mathbb{R}_{>0}$. 

\begin{prop}
Let $F$ be a non-trivial finite group and $H$ a finitely presented amenable group. Then $\mathrm{Sc}(F \wr H)$ is trivial if $H$ is one-ended and reduced to an explicit proper subgroup of $\mathbb{Q}_{>0}$ if $H$ is two-ended. 
\end{prop}

See Proposition \ref{prop:Sc(G)} for a more precise statement.
These facts rely on rigidity results obtained by Esking-Fisher-Whyte \cite{EFWII} for the case of $\mathbb{Z}$ (see \cite{MR2730576}) and on a recent work of the authors for the other cases \cite{QIwreath}.

\


\paragraph{Motivations and historical aspects.} An early question in geometric group theory, which can be found in \cite[\S 1.A']{GromovAsymptotic}, asks to which extend being quasi-isometric and being bi-Lipschitz equivalent are different. For instance, are two quasi-isometric graphs with bounded degree necessarily bi-Lipschitz equivalent? Partial positive answers were obtained for homogeneous trees \cite{MR1326733}, and next for hyperbolic groups \cite{Bogop}, before it was realized that a positive answer holds for every non-amenable graph with bounded degree in a strong way: 

\begin{thm}[\cite{MR1693847, MR1700742}]\label{thm:QIvsBil}
Every quasi-isometry between two non-amenable graphs with bounded degree lies at finite distance from a bijection.
\end{thm}

In \cite{MR1700742}, the author prove something stronger by characterizing precisely which quasi-isometries lie at finite distance from a bijection, thanks to the \emph{uniformly finite homology} introduced in \cite{BW}. Without going to much into the details, it is proved in \cite{BW} that a graph of bounded degree is amenable if and only if $H_0^{uf}(X) \neq 0$. On the other hand, Whyte proves in \cite{MR1700742} that a quasi-isometry between any two graphs of bounded degree $f : X \to Y$ lies at finite distance from a quasi-isometry if and only if $f_\ast ([X])=[Y]$, or equivalently, according to \cite{BW} (see also \cite[Theorem~7.6]{MR1700742}), if there exists some $C>0$ such that
$$\left| |A| - |f^{-1}(A)| \right| \leq C \cdot | \partial A| \text{ for all $A \subset Y$ finite.}$$
Thus, by using our terminology, Whyte proved that a quasi-isometry lies at finite distance from a bijection if and only if it is quasi-one-to-one, which turns out to be automatic when the graphs are not amenable.

But the techniques introduced by Whyte apply to arbitrary graphs of bounded degree, including amenable graphs, where they can be used in order to understand the obstruction of a quasi-isometry to lie at finite distance from a bijective quasi-isometry. The first results in this direction are due to Dymarz in \cite{MR2139686, MR2730576}. Although counterexamples for amenable graphs with bounded degree have been previously constructed, for instance in \cite{MR1616135, MR1616159}, the question for finitely generated groups remained open. For instance, the problem can be found in \cite[IV.B.46(vi)]{MR1786869}. First, Dymarz proved in \cite{MR2139686} that no analogue of Theorem \ref{thm:QIvsBil} holds for amenable groups: If $G$ is a finitely generated amenable group and $H \leq G$ a proper finite-index subgroup, then the inclusion $H \hookrightarrow G$ does not lie at finite distance from a bijection \cite{MR2139686}. (This observation is also a consequence of Lemma~\ref{lem:KappaWellDefined} and Proposition~\ref{prop:EasyKappa} below.) But this does not rule out the existence of a bi-Lipschitz equivalence $H \to G$ (that would not be at finite distance from the inclusion). Next, in \cite{MR2730576}, she exhibited the first examples of finitely generated groups that are quasi-isometric but not bi-Lipschitz equivalent, by considering lamplighter groups over~$\mathbb{Z}$. In her work, Dymarz focuses on quasi-isometries $f$ that are $k$-to-one, but explicitly uses the fact that $f_\ast([X])=k [Y]$, or equivalently that there exists some $C>0$ such that
$$\left| k |A| - |f^{-1}(A)| \right| \leq C \cdot | \partial A| \text{ for all $A \subset Y$ finite.}$$
Thus, by using our terminology, quasi-$\kappa$-to-one quasi-isometries are introduced and exploited in \cite{MR2139686, MR2730576} for $\kappa \in \mathbb{N}$. It is worth noticing that Proposition~\ref{prop:rational} is essentially contained in concrete examples studied in \cite{MR2139686, MR2730576}. Similar arguments are used in \cite{MR3318424} in order to prove that higher rank lamplighter groups provide finitely presented examples of groups that are quasi-isometric but not bi-Lipschitz equivalent.

Thus, so far quasi-$\kappa$-to-one quasi-isometries were only defined for $\kappa$ an integer and they were confined to a few family of groups. Our main motivation to develop a general theoretical framework comes from the following quasi-isometric classification (from which new examples of finitely generated groups that are quasi-isometric but not bi-Lipschitz equivalent can be deduced, see \cite{QIwreath} for more details):

\begin{thm}[\cite{QIwreath}]\label{thm:QIlamp}
Let $H$ be a finitely presented one-ended amenable group and $m,n \geq 2$ two integers. The lamplighter groups $\mathbb{Z} /n \mathbb{Z} \wr H$ and $\mathbb{Z}/m \mathbb{Z} \wr H$ are quasi-isometric if and only if there exist integers $k,r,s \geq 1$ such that $n=k^r$, $m=k^s$, and $r/s \in \mathrm{Sc}(H)$. 
\end{thm}

The necessity of a global theory dedicated to quasi-$\kappa$-to-one quasi-isometries comes from this theorem, and the goal of the article is precisely to initiate such a study. This subject being largely new, it raises many natural questions, a selection of which can be found in the last section.

\paragraph{Acknowledgements.} We thank Y. Cornulier and T. Dymarz for their comments on the first version of our manuscript.

\section{Metric measure spaces}\label{sec:MMspaces}

Our goal is to define our notion of measure-scaling quasi-isometries in a context that  includes locally compact groups, Riemannian manifolds and bounded degree graphs. 
In this section, we introduce our standing assumptions on a metric measure space that are suitable for such a definition, and we prove a few technical points that will be useful in the sequel.

\begin{definition}[Metric measure spaces]
A \textbf{metric measure space} is a triple $(X,d,\mu)$ where $X$ is a locally compact topological space, $\mu$ is a non-trivial, locally finite, Borel measure, and $d$ is a measurable distance\footnote{We do not ask the distance to be continuous: typically, the word distance on a compactly generated locally compact  non-discrete group is not continuous.} on $X$ such that balls are relatively compact.
\end{definition}\label{.}

The key examples are: graphs of bounded degree; Riemannian manifolds with bounded geometry; compactly generated locally compact groups, where $d=d_S$ is the left-invariant word distance associated to a compact symmetric generating subset $S$, and $\mu$ is a left Haar measure. All these examples share some important properties that are captured by the following standing assumptions.
\begin{definition}[Standard metric measure spaces]
Metric measure spaces $(X,d,\mu)$ satisfying the following properties will be called standard:
\begin{itemize}
\item $X$ has \textbf{bounded geometry}: for all $r\geq 1$, there exist $V_r\geq v_r>0$ 
such that
$$\forall x\in X, \; v_r\leq \mu(B(x,r))\leq V_r.$$
\item $X$ is $1$\textbf{-geodesic}: for every pair of points $x,y\in X$ there is a sequence $x=x_0,\ldots,x_n=y$ such that $d(x_{i-1},x_i)\leq 1$ for all $i$ and $d(x,y)=\sum_{i=1}^n d(x_{i-1},x_i)$.
\end{itemize}
\end{definition}\label{def:metricmeasurespaceStandard}

\begin{definition}[Thick subspaces] A subspace $Z\subset X$ is called \textbf{thick}  if it is a union of closed balls of radius $1$ (in particular it has positive measure). 
\end{definition}
\begin{definition}[Boundary] 
Given a subspace $A\subset X$ and $R>0$, we define the \textbf{$R$-boundary} of $A$ to be $$\partial A:=A^{+R}\cap
(X\setminus A)^{+R}.$$ 
\end{definition}

\begin{definition}[Amenability]\label{def:amenable}
The space $X$ will be called amenable if there exists a sequence $A_n$ of thick subsets with finite measure such that 
\[\lim_{n\to \infty}\frac{\mu(\partial_R A_n)}{\mu(A_n)}=0\]
for all $R>0$.
Such a sequence is called a F\o lner sequence.
\end{definition}

\begin{remark}\label{rem:geomAmenable}
Assume that the space $X=(G,d,\mu)$ is a compactly generated locally compact group, $d$ is a left-invariant word metric associated to a compact generating set and $\mu$ is a left Haar measure. Then the metric measure space $X$ is amenable in the above sense if and only if  the group $G$ is amenable and unimodular (see \cite[Proposition~11.11]{T} or \cite[Proposition~4.F.5]{CornHarpe}).
\end{remark}

\begin{remark}\label{remark:Boundary}
Note that 
$$\partial_R A=\left(A^{+R}\setminus A\right)\cup \left(\bar{A}^{+R}\setminus \bar{A}\right)$$
where $\bar{A}$ denotes the complement of $A$ and $A^{+R}$ its $R$-neighborhood. This implies in particular that
\[(\partial_R A)^{+1}\subset \partial_{R+1} A.\]
\end{remark}
In practice it is enough to consider the case $R=2$ as shown by the following easy fact. We use throughout the following convenient convention: given a ball $B(x,r)$ and $\lambda>0$, the notation ``$\lambda B(x,r)$'' means $B(x,\lambda r)$.

\begin{fact}\label{fact:EasyOne}
Let $X$ be a standard metric measure space and let $A$ be a measurable subset of finite measure.  We denote $\bar{A}=X\setminus A$. 
Then for all $R\geq 2$, \begin{itemize}
\item[(i)] if $A$ is thick, then $\mu(A^{+R}) \leq (V_{R+3}/v_1)\cdot \mu(A);$
\item[(ii)] $\mu(A^{+R}\setminus A) \leq (V_{2+R}/v_1)\cdot \mu(\partial_2 A);$
\item[(iii)] $\mu(\bar{A}^{+R}\setminus \bar{A}) \leq (V_{2+R}/v_1)\cdot \mu(\partial_2 A);$
\item[(iv)] $ \mu(\partial_R A)\leq (2V_{2+R}/v_1)\cdot \mu(\partial_2 A).$
\end{itemize}
\end{fact}

\begin{proof} Let us prove (i).
Let $\mathbf{B}$ be a maximal set of disjoint $1$-balls contained in $A$. Let $y\in A$. Since $A$ is thick, $y$ belongs to a ball $B_0$ of radius $1$ contained in $A$. Since $\mathbf{B}$ is maximal, $B_0$ must interest an element $B$ of $\mathbf{B}$ non-trivially. Hence $y$ lies in $3B$. It follows that $A$ is contained in $\cup_{B\in \mathbf{B}} 3B$, and therefore that $A^{+R}$ is contained in $\cup_{B\in \mathbf{B}} (3+R)B$. We deduce that \[\mu(A^{+R})\leq \frac{V_{R+3}}{v_{1}}\sum_{B\in \mathbf{B}}\mu(B)\leq \frac{V_{R+3}}{v_{1}}\mu(A),\]
and the first statement is proved.

Observe that (iv) results from (ii) and (iii) as a consequence of Remark~\ref{remark:Boundary}.
The proofs of (ii) and (iii) are identical, so we only prove (ii). 
We let $x\in A^{+R}\setminus A$, consider a minimal $1$-geodesic path joining $x$ to $A$ and let $y$ be the last point along that geodesic that remains outside $A$. Since $d(y,A)\leq 1$, necessarily $B(y,1)$ is contained in $\partial_2A$. Now let $\mathbf{B}$ be a maximal set of disjoint $1$-balls contained in $\partial_2 A$. By maximality, $B(y,1)$ must meet a ball  $B$ of $\mathbf{B}$. By triangular inequality, we deduce from $d(x,y)\leq d(x,A)\leq R$ that $x$ belongs to $(R+2)B$. We conclude that 
\[\mu(A^{+R}\setminus A) \leq  \frac{V_{R+2}}{v_{1}}\sum_{B\in \mathbf{B}}\mu(B)\leq \frac{V_{R+2}}{v_{1}}\mu(\partial_2 A).\]
And (ii) is proved. 
\end{proof}

\begin{remark}
Given a graph $Z$ and a set of vertices $S \subset Z$, a common definition of  \emph{boundary} $\partial S$ of $S$ is defined as $S^{+1}\setminus S$. If the graph has bounded degree, it is easy to see  $|\partial S|$, $|\partial_1 S|$ and $|\partial_2 S|$ only differ by a multiplicative constant.
It follows that Fact~\ref{fact:EasyOne} holds in that case for all $R\geq 1$. But this fails in general as shown by the following example: let $X=\bigcup_{n\in \mathbb{Z}} [2n,2n+1]$ equipped with the subspace distance induced from the usual distance on $\mathbb{R}$. Equipped with the Lebesgue measure, this space is a standard metric measure space. But the $1$-boundary of any set of the form $A=[0,2n+1]$ is reduced to $\{-1\}\cup\{0\} \cup\{2n+1\}\cup\{2n+2\}$ and therefore has zero measure, while the $2$-boundary has positive measure. 
\end{remark}

\begin{definition}
Given $C\geq 1$ and $K \geq 0$, a map $f : X \to Y$ between two metric spaces is a \emph{$(C,K)$-quasi-isometry} if
$$\frac{1}{C} \cdot d(x,y) - K \leq d(f(x),f(y)) \leq C \cdot d(x,y) + K \text{ for all $x,y \in X$},$$
and if every point in $Y$ is within $K$ from $f(X)$. 
\end{definition}
\begin{fact}\label{fact:EasyTwo}
Let $X,Y$ be two metric spaces, let $f : X \to Y$ a $(C,K)$-quasi-isometry. For every subset $B\subset Y$, and every $R\geq 0$, 
\begin{itemize}
\item[(i)]   $f^{-1}(B)^{+R}\subset f^{-1}(B^{+(CR+K)});$
\item[(ii)]  $\partial_R f^{-1}(B)\subset f^{-1}(\partial_{CR+K}B).$
\end{itemize}
\end{fact}
\begin{proof}
(i) is immediate from the definition of $(C,K)$-quasi-isometry, and (ii) results from (i) applied to $B$ and its complement.
\end{proof}
\begin{fact}\label{fact:EasyOne'}
Let $(X,\mu),(Y,\nu)$ be two standard metric measure spaces, let $f : X \to Y$ a $(C,K)$-quasi-isometry and let $R\geq 2$. There exist $L,L',L''>0$ such that 
\begin{itemize}
\item[(i)] 
for all thick subsets $A\subset X$ and $G\subset Y$,  \[\nu(f(A))\leq L \cdot \mu(A) \;{\text and }\; \mu(f^{-1}(G))\leq L \cdot \nu(G);\]
\item[(ii)] for every measurable subset $G\subset Y$, and every $R\geq 2$, we have
 \[\mu(f^{-1}(\partial_R G))\leq L \cdot \nu(\partial_2 G);\]
 \item[(iii)] for every measurable subset $G\subset Y$, and every $R\geq 2$, we have
 \[\mu(\partial_Rf^{-1}(G))\leq L''\cdot \nu(\partial_2 G).\]
 \end{itemize}
 \end{fact}

\begin{proof}
The two inequalities in (i) are proved in the same way so we focus on the second one.
We let $\mathbf{B}$ be a maximal set of disjoint $1$-balls contained in $A$. By maximality of $\mathbf{B}$, every point of $A$ belongs to a $1$-ball intersecting a ball in $\mathbf{B}$, so $A$ must be contained in $\cup_{B\in \mathbf{B}} 3B$. On the other hand the preimage of a ball of radius $3$, if not empty has diameter at most $3C+K$, hence is contained in a ball of radius $6C+K$. We deduce that 
\[\mu(f^{-1}(A))\leq \sum_{B\in \mathbf{B}}\mu(f^{-1}(3B))\leq \frac{v^X_{6C+K}}{v^Y_1}\sum_{B\in \mathbf{B}}\nu(B)\leq  \frac{v^X_{6C+K}}{v^Y_1}\nu(A),\]
and (i) is proved with $L=v^X_{6C+K}/v^Y_1.$

\noindent To prove (ii), apply the second inequality of (i) to $A=(\partial_R G)^{+1}$, use the fact that $A\subset \partial_{R+1} G$ and finally conclude with Fact \ref{fact:EasyOne}(iv). 

\noindent Finally, (iii) follows from the combination of Fact \ref{fact:EasyTwo}(ii) and (ii).
 \end{proof}

Amenability in the sense of Definition \ref{def:amenable} turns out to be invariant under quasi-isometry: indeed quasi-isometries between standard metric measure spaces are large scale equivalences in the sense of \cite{T} and so this follows from \cite[Theorem 8.1]{T} (in the special case where $\phi$ is a positive constant). For completeness we give a short proof below.

\begin{prop}
Amenability is invariant under quasi-isometry among standard metric measure spaces.
\end{prop}
\begin{proof}
Assume that $f : X \to Y$ is a $(C,K)$-quasi-isometry between two standard metric measure spaces $(X,\mu)$ and $(Y,\nu)$, and let $G_n$ be a F\o lner sequence of $Y$. We claim that $A_n:= f^{-1}(G_n^{+K})^{+1}$ defines a F\o lner sequence of $X$. 

Let $\mathbb{B} \subset A_n$ be such that $\{B(b,1) \mid b \in \mathbb{B} \}$ is a maximal collection of pairwise disjoint balls in $A_n$. We have
$$\mu(A_n) \geq \sum\limits_{b \in \mathbb{B}} \mu(B(b,1)) \geq c \cdot \sum\limits_{b \in \mathbb{B}} \nu(B(f(b),3C+2K)) \geq c \cdot \nu \left( \bigcup\limits_{b \in \mathbb{B}} B(f(b),3C+2K) \right)$$
where $c:= v^X_1 / V^Y_{3C+2K}$. Observe that $G_n \subset \bigcup_{b \in \mathbb{B}} B(f(b),3C+2K)$. Indeed, if $g \in G_n$ then there exists some $a \in X$ such that $d(f(a),g)) \leq K$. Necessarily, $a$ belongs to $A_n$, and there must exist some $b \in \mathbb{B}$ such that $B(a,1)$ and $B(b,1)$ intersect by maximality of our collection, hence $d(a,b) \leq 3$. So we have
$$d(f(b),g) \leq d(f(b),f(a)) + d(f(a),g) \leq 3C+2K,$$
i.e. $g \in B(f(b),3C+2K)$. We conclude that 
\begin{equation}\label{eq:subset}
\mu(A_n)\geq c \cdot \mu(G_n).
\end{equation}
Next, observe that $A_n \subset f^{-1}(G_n^{+(C+2K)})$ according to Fact~\ref{fact:EasyTwo}(i), so we deduce from Fact \ref{fact:EasyOne'}(iii) that there exists a constant $Q$ such that 
\[\mu(\partial_2 A_n)\leq Q \cdot \nu(\partial_2 G_n^{+(C+2K)}). \]
But $\partial_2 G_n^{+(C+2K)}\subset \partial_{C+2K+2}G_n$, so we deduce from Fact \ref{fact:EasyOne}(iv) that there exists a constant $Q'$ such that
 \begin{equation}\label{eq:partialsebset}
\mu(\partial_2A_n)\leq Q' \cdot \mu(\partial_2 G_n).
\end{equation}
We deduce from (\ref{eq:subset}) and (\ref{eq:partialsebset}) that $(A_n)$ is a F\o lner sequence, as required.
\end{proof}

\section{Scaling quasi-isometries}\label{section:QuasiToOne}

We are now ready to define measure-scaling quasi-isometries between standard metric measure spaces. Beyond providing this definition, we shall prove in this section that measure-scaling quasi-isometries are stable under composition.  
\begin{definition}\label{def:measurekappa}
Let $\kappa>0$. A measurable map between measure spaces $f:(X,\mu)\to (Y,\nu)$ is \emph{measure $\kappa$-to-one} if $f\ast \mu= \kappa\nu$: in other words $\mu(f^{-1}(A))=\kappa \nu(A)$ for all measurable subsets $A$.
\end{definition}
\noindent If the measures on $X$ and $Y$ are the counting measures, then $\kappa$ must be an integer, and being measure $\kappa$-to-one amounts to being surjective and $\kappa$-to-one in the usual sense: the pre-image a point has size $\kappa$.
Examples of measure $\kappa$-to-one map are isomorphisms between locally compact groups equipped with left Haar measures.
Our goal is to give an approximate version of being measure $\kappa$-to-one for metric measure spaces.
\begin{definition}\label{def:quasione}
Let $f:(X,\mu)\to (Y,\nu)$ be a measurable map between two standard metric measure spaces $X,Y$ and let $\kappa>0$. Then $f$ is \emph{quasi-$\kappa$-to-one} if there exists a constant $C>0$ such that for all thick subspace $A\subset Y$, we have
\[\left| \kappa \nu(A)- \mu(f^{-1}(A)) \right|\leq C\nu(\partial_2 A).\]
A quasi-isometry between  standard metric measure spaces is called a \emph{measure-scaling quasi-isometry of (measure-)scaling factor} $\kappa$ if it is quasi-$\kappa$-to-one.
\end{definition}

\noindent 
We emphasize that it is not necessary to consider the $2$-boundary $\partial_2$ in the above definition, but this can be done without loss of generality according to Fact~\ref{fact:EasyOne}. Obvious examples of quasi-$\kappa$-to-one maps are measure $\kappa$-to-one maps.

\begin{remark}
 By a result of Whyte, any quasi-one-to-one quasi-isometry between bounded degree graphs lies at bounded distance from a bijective quasi-isometry  \cite[Theorems~A and~C]{MR1700742} (see also \cite[Theorems~5.3 and~5.4]{MR2730576}). We shall see a simple generalization of this fact when $\kappa$ is rational (see \S \ref{section:graphs}).
\end{remark}

\begin{remark}\label{rem:scalingNonAmen}
Note that when $X$ is non-amenable, Fact \ref{fact:EasyOne'} shows that this condition is essentially empty: any quasi-isometry being quasi-$\kappa$-to-one for any $\kappa>0$. This is in sharp contrast with the amenable case as shown by the following lemma.
\end{remark}

\begin{lemma}\label{lem:KappaWellDefined}
Let $f : (X,\mu) \to (Y,\nu)$ be a measurable map between two standard metric measure spaces. Assume that $X$ is amenable. If $f$ is both quasi-$\kappa_1$-to-one and quasi-$\kappa_2$-to-one for some $\kappa_1,\kappa_2>0$, then $\kappa_1=\kappa_2$.
\end{lemma}

\begin{proof}
Let $(A_n)$ be a F\o lner sequence in $X$. It follows from the definition of being quasi-$\kappa_1$-to-one that $\mu(f^{-1}(A_n)) / \nu(A_n) \to \kappa_1$ as $n \to + \infty$. Similarly, one gets $\mu(f^{-1}(A_n)) / \mu(A_n) \to \kappa_2$ as $n \to + \infty$, hence $\kappa_1=\kappa_2$. 
\end{proof}

\noindent
Our next observation shows how being quasi-$\kappa$-to-one behaves under composition.  

\begin{prop}\label{prop:EasyKappa}
Let $(X,\mu),(Y,\nu),(Z,\sigma)$ be three standard metric measure spaces, $\kappa_1,\kappa_2>0$ two real numbers, and $f,h:X\to Y$ and $g:Y\to Z$  three quasi-isometries. 
\begin{itemize}
\item[(i)] If $f, h$ are at bounded distance and if $f$ is quasi-$\kappa_1$-to-one, then $h$ is also quasi-$\kappa_1$-to-one.
\item[(ii)] If $f$ and $g$ are respectively quasi-$\kappa_1$-to-one and quasi-$\kappa_2$-to-one, then $g\circ f$ is quasi-$\kappa_1\kappa_2$-to-one. 
\item[(iii)] If $\bar{f}$ is a quasi-inverse of $f$ and if $f$ is quasi-$\kappa_1$-to-one, then $\bar{f}$ is quasi-$(1/\kappa_1)$-to-one.  
\end{itemize}
\end{prop}

\noindent

\begin{proof}[Proof of Proposition \ref{prop:EasyKappa}.]
To avoid clumsy expressions, we shall adopt the following useful notation: $F\lesssim_{\alpha_1,\alpha_2,\ldots}G$ means that there exists $C>0$ only depending on the parameters $\alpha_1,\alpha_2,\ldots$ such that $F\leq CG$.

\noindent We start proving (i).
First, assume that $f,h$ are at distance $\leq Q$, and that $f$ is quasi-$\kappa_1$-to-one, i.e. there exists some $C \geq 0$ such that 
$$\left | \kappa_1\nu(A) -  \mu(f^{-1}(A)) \right| \leq C \cdot \nu(\partial_2 A ) \text{ for all  finite measure thick subset $A \subset Y$.}$$
Observe that, because the distance between $f$ and $h$ is $\leq Q$, we have $f^{-1}(B) \subset h^{-1}(B^{+Q})$ and $h^{-1}(B) \subset f^{-1}(B^{+Q})$ for every  finite measure thick subset  $B \subset Y$. This observation will be used several times in the sequel. Given a finite measure thick subset $A \subset Y$, we have
\begin{equation}\label{eq:Quasi}
\left| \kappa_1 \nu(A) - \mu(h^{-1}(A)) \right| \leq \left| \kappa_1 \nu(A) -  \mu(f^{-1}(A)) \right| +  \left| \mu(f^{-1}(A)) -\mu(h^{-1}(A)) \right|.
\end{equation}
In the right hand side, we know how to control the left term because $f$ is quasi-$\kappa_1$-to-one, so it remains to control the right term. On the one hand, we have
$$\begin{array}{lcl}\mu(h^{-1}(A))-\mu(f^{-1}(A)) & \leq & \mu(f^{-1}(A^{+Q}))- \mu(f^{-1}(A)) \leq \mu \left( f^{-1}(A^{+Q}) \backslash f^{-1}(A) \right) \\ \\ & \leq & \mu\left( f^{-1}(A^{+Q} \backslash A) \right)\leq \mu(f^{-1}(\partial_Q A))\lesssim_{Q,f}  \nu(\partial_2 A). \end{array}$$
where the last inequality is a consequence of Fact \ref{fact:EasyOne'}(ii).

 On the other hand, we have similarly
$$\begin{array}{lcl} \mu(f^{-1}(A)) -\mu(h^{-1}(A)) & \leq & \mu(h^{-1}(A^{+Q}))-\mu(h^{-1}(A)) \leq \mu\left( h^{-1}(A^{+Q} \backslash A )\right)\lesssim_{Q,h}  \nu(\partial_2 A). \end{array}$$

Thus, we have proved that $\left| \mu(f^{-1}(A)) -\mu(h^{-1}(A)) \right| \lesssim_{f,g}\nu(\partial_2 A)$. Combining this observation with the inequality (\ref{eq:Quasi}), we get
$$\left| \kappa_1 \nu(A) -\mu(h^{-1}(A)) \right| \lesssim_{f,h} \nu(\partial_2 A).$$
This concludes the proof that $h$ is quasi-$\kappa_1$-to-one.

\medskip \noindent
Next we prove (ii). Assume that $f$ and $g$ are respectively quasi-$\kappa_1$-to-one and quasi-$\kappa_2$-to-one. So there exist $P,Q \geq 0$ such that
$$\left| \kappa_1\nu(A) -  \mu(f^{-1}(A)) \right| \leq P \cdot \nu(\partial_2 A )\text{ and } \left| \kappa_2 \sigma(B) -  \nu(g^{-1}(B)) \right| \leq Q \cdot \sigma(\partial_2 B)$$
for all finite measure thick subsets $A \subset Y$, $B \subset Z$. Given a finite measure thick subset $A \subset Z$, we have
$$\begin{array}{lcl} \displaystyle \left| \kappa_1\kappa_2\sigma(A) - \mu((g \circ f)^{-1}(A)) \right| & \leq & \displaystyle \kappa_2 \left| \kappa_1 \sigma(A) -  \nu(g^{-1}(A)) \right|+ \left| \kappa_2 \nu(g^{-1}(A)) - \mu(f^{-1}(g^{-1}(A))) \right| \\ \\ & \leq & \displaystyle P\kappa_2 \cdot \sigma(\partial_2 A) + Q \cdot \nu( \partial_2 g^{-1}(A)) \end{array}$$
As a consequence of Fact \ref{fact:EasyOne'}(iii), $\nu(\partial_2 g^{-1}(A)) \lesssim_g \sigma(\partial_2 A)$. Therefore,
$$\left| \kappa_1\kappa_2\sigma(A) -  \mu((g \circ f)^{-1}(A)) \right| \lesssim_{f,g} \sigma(\partial_2 A),$$
concluding the proof that $f\circ g$ is quasi-$\kappa_1\kappa_2$-to-one.

\medskip \noindent
We turn to the proof of (iii). Assume that $f$ is quasi-$\kappa_1$-to-one, i.e. there exists $Q \geq 0$ such that
$$\left| \kappa_1 \nu(A) - \mu(f^{-1}(A)) \right| \leq Q \cdot \nu(\partial_2 A) \text{ for all finite measure thick subset $A \subset Y$},$$
and let $\bar{f}$ be a quasi-inverse of $f$. We fix two constants $C,K \geq 0$ such that $f, \bar{f}$ are $(C,K)$-quasi-isometries and $\bar{f}\circ f, f \circ \bar{f}$ are within $K$ from identities. Given a finite measure thick subset $A \subset X$, observe that
$$\bar{f}^{-1}(A) \subset f(A)^{+K} \subset f(A)^{+K+1} \subset \bar{f}^{-1}(A^{+C(3K+1)}).$$
For the first inclusion, notice that, given an $x \in \bar{f}^{-1}(A)$, we have $d(x,f(\bar{f}(x))) \leq K$ where $\bar{f}(x) \in A$, hence $d(x,f(A)) \leq K$. For the third inclusion, notice that, given an $x \in f(A)^{+K+1}$, there exists some $a \in A$ such that $d(x,f(a)) \leq K+1$, hence
$$\begin{array}{lcl} d(\bar{f}(x),a) & \leq & C d(f(\bar{f}(x)), f(a)) +CK \leq C (d(f(\bar{f}(x)),x)+d(x,a)) +CK \\ \\ & \leq & C( K+ K+1)+CK, \end{array}$$
i.e. $x \in \bar{f}^{-1}(A^{+C(3K+1)})$. From these inclusions, we deduce that
\begin{equation}\label{equationOne}
\left| \nu(\bar{f}^{-1}(A)) - \nu(f(A)^{+K}) \right| \leq \nu\left( \bar{f}^{-1} \left( A^{C(3K+1)} \backslash A \right) \right)\leq\nu(\partial_{C(3K+1)}\bar{f}^{-1}(A)) \lesssim_{C,K}\mu(\partial_2 A),
\end{equation}
where the last inequality follows from Fact \ref{fact:EasyOne'}(iii). 
Note that \[\partial_2 f(A)^{+K}\subset \partial_{C(3K+1)+2}\bar{f}^{-1}(A).\] Hence applying once again Fact \ref{fact:EasyOne'}(iii), we obtain
\begin{equation}\label{EquationTwo}
\nu\left(\partial_2 \left( f(A)^{+K} \right) \right)  \lesssim_{C,K} \mu(\partial_2 A).
\end{equation}
Next, observe that
$$ A \subset f^{-1} \left( f(A)^{+K} \right) \subset A^{+2CK}.$$
The first inclusion is clear. For the second inclusion, notice that, given an $x \in f^{-1}(f(A)^{+K})$, there exists some $a \in A$ such that $d(f(x),f(a)) \leq K$, hence $d(x,a) \leq C d(f(x),f(a)) +CK \leq 2CK$, i.e. $x \in A^{+2CK}$. We deduce from these inclusions that
\begin{equation}\label{equationThree}
\left| \mu(A) -\mu( f^{-1} ( f(A)^{+K}) ) \right| \leq \mu\left( A^{+2CK}\backslash A \right) \lesssim_{C,K} \mu(\partial_2 A),
\end{equation}
where we have used Facts \ref{fact:EasyOne}(iii). We have:
\begin{eqnarray*}
\left| \frac{1}{\kappa_1} \mu(A) - \nu(f(A)^{+K})\right| & \leq & \frac{1}{\kappa_1} \left| \mu(A) -\mu( f^{-1} ( f(A)^{+K}) ) \right|  + \left| \frac{1}{\kappa_1} \mu( f^{-1} ( f(A)^{+K}) )- \nu(f(A)^{+K})\right| \\
& \lesssim_{C,K} &\mu(\partial_2 A)+ \nu(\partial_2f(A)^{+K})) \\
& \lesssim_{C,K}& \mu(\partial_2 A),
\end{eqnarray*}
where the second inequality follows from  (\ref{equationThree}), and from our assumption on $f$, while the third inequality results from  (\ref{EquationTwo}).
So finally, combining the latter inequality and equation (\ref{equationOne}), we have

\begin{eqnarray*}
\left| \frac{1}{\kappa_1} \mu(A) - \nu(\bar{f}^{-1}(A)) \right| & \leq & \left| \frac{1}{\kappa_1} \mu(A) - \nu(f(A)^{+K})\right| +\left| \nu(f(A)^{+K})-\nu(\bar{f}^{-1}(A))\right|\\
& \lesssim_{C,K} &  \mu(\partial_2 A),
\end{eqnarray*}
concluding the proof of the fact that $\bar{f}$ is quasi-$(1/\kappa_1)$-to-one.
\end{proof}

\section{The case of graphs and rational  scaling factors}\label{section:graphs}
We now focus on graphs, for which we have the following important result due to Whyte.
\begin{prop}\label{prop:QIdistBij}
Let $f:X_1\to X_2$ be a quasi-isometry between two graphs with bounded degree. Then $f$ is at bounded distance from a bijection if and only if it is quasi-one-to-one.
\end{prop}

\begin{proof}
We briefly recall the definition of the first uniformly finite homology group $H_0^{uf}$ (originally introduced by Block and Weinberger in \cite{BW}). 
For a bounded degree graph $X$, let $C_0^{uf}$ denote the abelian group of formal sums of the form
\[c = \sum_{x\in X}a_x x\]
where $a_x$ is a bounded family of integers.
 Let $C_1^{uf}$ be abelian group of formal sums
\[c = \sum_{(x,y)\in X\times X}a_{x,y} (x,y),\]
where $a_{x,y}$ is again a bounded family of integers that are zero except in a bounded neighborhood of the diagonal.
The
boundary map $\partial: C_1^{uf}\to C_0^{uf}$ is defined as usual as  
$\partial (x,y)=y-x.$ We denote by $H_0^{uf}=C_0^{uf}(X)/\partial(C_1^{uf}(X))$ the corresponding zeroth homology group.
The fundamental class of $X$ is denoted by $[X]=\sum_{x\in X}x$.

\medskip \noindent
By \cite[Theorem A]{MR1700742}, $f$ is at bounded distance from a bijection if and only if $f_*([X_1])=[X_2]$ in $H_0^{uf}(X_2)$, where
$f_*([X_1])=\sum_{x\in X_2} |f^{-1}(x)|x$. In other words, this condition says that 
$c=\sum_{x\in X_2} (|f^{-1}(x)|-1)x$ satisfies $[c]=0$ in $H_0^{uf}(X_2)$, which amounts to saying, according to \cite[Theorem C]{MR1700742}, that there exists a constant $C>0$ such that, for all finite subset $A\subset X_2$,
\[\left|\sum_{x\in A}c_x\right|\leq C|\partial A|.\]
This concludes the proof of our proposition.
\end{proof}

\noindent
In case $\kappa$ is rational, we have the following equivalent formulations of Definition~\ref{def:quasione}.

\begin{prop}\label{prop:kappa}
Let $m,n \geq 1$ be natural integers and $f:X\to Y$ a quasi-isometry between two graphs with bounded degree. The following statements are equivalent:
\begin{itemize}
\item[(i)] $f$ is quasi-$(m/n)$-to-one;
\item[(ii)] the map $\iota \circ f\circ \pi$ is at bounded distance from a bijection, where $\pi : X\times \mathbb{Z}/n\mathbb{Z} \twoheadrightarrow X$ is the canonical embedding and $\iota : Y \hookrightarrow Y\times \mathbb{Z}/m \mathbb{Z}$ the canonical projection.
\item[(iii)] there exist a partition $\mathcal{P}_X$ (resp. $\mathcal{P}_Y$) of $X$ (resp. of $Y$) with uniformly bounded pieces of size $m$ (resp. $n$) and a bijection $\psi:\mathcal{P}_X\to \mathcal{P}_Y$ such that $f$ is at bounded distance from a map $g : X \to Y$ satisfying $g(P) \subset \psi(P)$ for every $P \in \mathcal{P}_X$.
\end{itemize}
\end{prop}

\noindent
The following preliminary lemma will be needed in our proof of the proposition.

\begin{lemma}\label{lem:Partition}
For every locally finite connected graph $X$ and every integer $k\geq 1$, there exists a partition $\mathcal{P}_X$ of $X$ whose pieces have constant size $k$ and diameter at most $2(k-1)$.
\end{lemma}

\begin{proof}
Our lemma will be essentially a consequence of the following observation:

\begin{claim}\label{claim:Subtree}
Let $k \geq 1$ be an integer and $T$ a finite tree of diameter $>k$. There exists a subset $S \subset T$ of $k$ vertices that has diameter $\leq 2(k-1)$ and such that $T \backslash S$ is connected.
\end{claim}
\begin{figure}
\begin{center}
\includegraphics[width=0.5\linewidth]{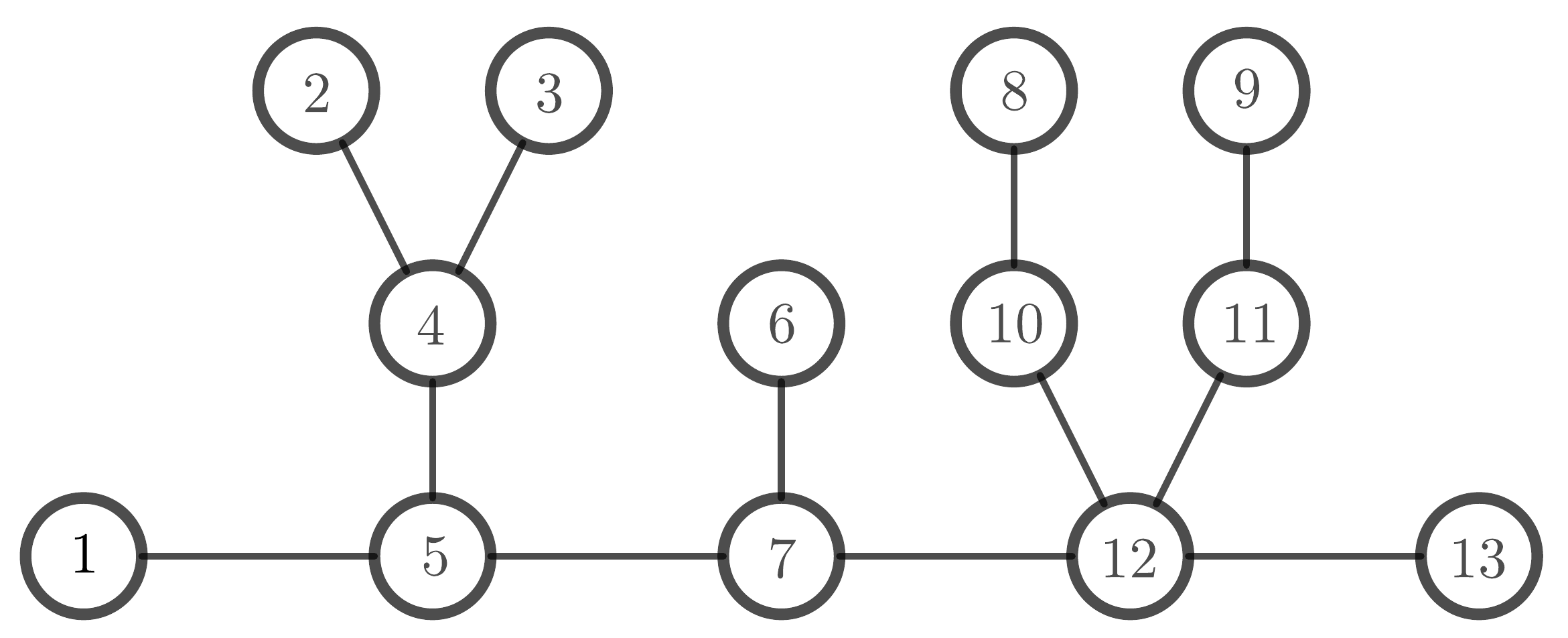}
\caption{Numbering a finite tree.}
\label{Arbre}
\end{center}
\end{figure}

\noindent
Fix two vertices in $T$ at distance $\mathrm{diam}(T)$, and let $u_0,\ldots, u_d$ denote the unique geodesic between them. For every $1 \leq i \leq d-1$, let $T_i$ denote the connected component of $T \backslash \{u_0, \ldots, u_{i-1},u_{i+1},\ldots, u_d\}$ that contains $u_i$; we denote by $d_i$ the maximal distance from $u_i$ to a vertex in $T_i$. Now, we number the vertices of $T$ in the following way:
\begin{itemize}
	\item number $u_0$ as the first vertex;
	\item number the vertices in $T_1$ at distance $d_1$ from $u_1$ in an arbitrary number, do the same thing with the vertices at distance $d_1-1$ from $u_1$, and so on, until numbering $u_1$;
	\item iterate the numbering with $T_2$, then $T_3$, and so on until $T_{d-1}$.
	\item Finally, number $u_d$ as the last vertex of $T$.
\end{itemize}
See for instance Figure \ref{Arbre}. Let $S$ denote the first $k$ vertices with respect to this numbering. By construction, $T \backslash S$ is connected and $S$ is contained in $\{u_0\} \cup T_1 \cup \cdots \cup T_{k-1}$. Observe that
$$\begin{array}{lcl} d(v,u_{k-1}) & = & d(v,u_d)-d(u_{k-1},u_d) \leq \mathrm{diam}(T) - d(u_{k-1},u_d) \\ \\ & \leq & d(u_0,u_d)-d(u_{k-1},u_d)=d(u_0,u_{k-1})=k-1 \end{array}$$
for every vertex $v \in \{u_0\} \cup T_1 \cup \cdots \cup T_{k-1}$. Therefore, $\{u_0\} \cup T_1 \cup \cdots \cup T_{k-1}$, and a fortiori $S$, has diameter at most $2(k-1)$. This concludes the proof of our claim.

\medskip \noindent
Now, let us go back to our graph $X$. Fix a spanning tree $T \subset X$ and let $T(1),T(2), \ldots$ be an increasing sequence of subtrees covering $T$. For every $i \geq 1$, applying Claim~\ref{claim:Subtree} iteratively provides a partition $\mathcal{P}_{T(i)}$ of $T(i)$ whose pieces have constant size $k$ and diameter at most $2(k-1)$ (in $T(i)$, and a fortiori in $X$). Since $X$ is locally finite, a diagonal argument shows that $(\mathcal{P}_{T(i)})$ subconverges to some partition $\mathcal{P}_X$, i.e. $(\mathcal{P}_{T(i)})$ admits a subsequence of partitions that are eventually constant (to $\mathcal{P}_X$) on each finite set. This is the partition we are looking for.
\end{proof}

\begin{proof}[Proof of Proposition \ref{prop:kappa}.]
We begin by proving the implication $(iii) \Rightarrow (i)$. Let $A \subset Y$ be a finite set of vertices, and let $A^+$ denote the union of all the pieces in $\mathcal{P}_Y$ that contain at least one point of $A$. If $A^+$ is a union of $k$ pieces of $\mathcal{P}_Y$, then $f^{-1}(A^+)$ is a union of $k$ pieces of $\mathcal{P}_X$, hence $|f^{-1}(A^+)|/m = k = |A^+|/n$. So we have
$$\begin{array}{lcl} \displaystyle \left| \frac{m}{n} |A|-  |f^{-1}(A)| \right| & \leq & \displaystyle \frac{m}{n} \left| |A^+|- |A| \right| + \left| \frac{m}{n} |A^+| -  |f^{-1}(A^+) | \right| + \left| |f^{-1}(A^+)|-|f^{-1}(A)| \right| \\ \\ & \leq & \displaystyle \frac{m}{n} \left| A^+\backslash A \right| +  \left| f^{-1}(A^+ \backslash A) \right| \leq \left( \frac{m}{n} + m \right) |A^+ \backslash A|  \end{array}$$ 
where the last inequality is justified by the fact that, for every finite set of vertices $B \subset Y$, we have $|f^{-1}(B) | \leq m |B|$. Indeed, $B$ is contained in at most $|B|$ pieces of $\mathcal{P}_Y$, so $f^{-1}(B)$ must be contained in at most $|B|$ pieces of $\mathcal{P}_X$, hence $|f^{-1}(B)| \leq m |B|$. Because the pieces of $\mathcal{P}_Y$ are uniformly bounded, there exists some $K \geq 0$ such that $A^+ \subset A^{+K}$. We deduce from Fact \ref{fact:EasyOne} that
$$\left| \frac{m}{n} |A|-  |f^{-1}(A)| \right| \leq \frac{m}{n}(n+1) |A^+\backslash A | \leq \frac{m}{n}(n+1) |A^{+K}\backslash A| \leq \frac{m}{n}C(n+1) |\partial A|$$
for some constant $C \geq 0$ that does not depend on $A$. Thus, we have proved that $f$ is quasi-$(m/n)$-to-one, as desired.

\medskip \noindent
The implication $(i) \Rightarrow (ii)$ can be deduced from $(iii) \Rightarrow (i)$. Indeed, we deduce from the latter implication that $\iota$ and $\pi$ are respectively quasi-$(1/m)$-to-one and quasi-$n$-to-one. Therefore, Proposition \ref{prop:EasyKappa}(ii) implies that $\iota \circ f \circ \pi$ is quasi-one-to-one, and $(iii)$ follows from Proposition \ref{prop:QIdistBij}. 

\medskip \noindent
Now, we want to prove the implication $(ii) \Rightarrow (i)$. Given a finite subset $A \subset Y$, we set $A^+:= A \times \mathbb{Z}/m \mathbb{Z} \subset Y \times \mathbb{Z}/m \mathbb{Z}$. Observe that
$$|A^+|= m |A| \text{ and } \left| f^{-1}(A) \right| = \frac{1}{n} \left| (\iota \circ f \circ \pi)^{-1}(A^+) \right|,$$
so we have
$$\left| \frac{m}{n} |A| - |f^{-1}(A)| \right| = \frac{1}{n} \left| |A^+| - |(\iota \circ f \circ \pi)^{-1}(A^+) \right| \leq \frac{Q}{n} \left| \partial_{Y \times \mathbb{Z}/m \mathbb{Z}} A^+ \right|,$$
where $Q \geq 0$ is a constant that does not depend on $A$, because it follows from the fact that $\iota \circ f \circ \pi$ is at bounded distance from a bijection that it is quasi-one-to-one (according to Proposition \ref{prop:EasyKappa}(i)). By noticing that $\partial_{Y \times \mathbb{Z}/m\mathbb{Z}} A^+ = \partial A \times \mathbb{Z}/m \mathbb{Z}$, we deduce that
$$\left| \frac{m}{n} |A| - |f^{-1}(A)| \right| \leq \frac{Qm}{n} | \partial A|,$$
concluding the proof that $f$ is quasi-$(m/n)$-to-one.

\medskip \noindent
Finally, we turn to the implication $(i) \Rightarrow (iii)$. According to Lemma \ref{lem:Partition}, there exists a partition $\mathcal{P}_X$ (resp. $\mathcal{P}_Y$) of $X$ (resp. of $Y$) with uniformly bounded pieces of constant size $m$ (resp. $n$). Let $D$ denote a uniform upper bound on the diameters of the pieces in $\mathcal{P}_X$, $\mathcal{P}_Y$. We endow $\mathcal{P}_X$ (resp. $\mathcal{P}_Y$) with a graph structure by connecting two pieces with an edge whenever they contain two adjacent vertices. For every $P \in \mathcal{P}_X$ (resp. $Q \in \mathcal{P}_Y$), we fix a basepoint $x_P \in P$ (resp. $y_Q \in Q$). By applying the implication $(iii) \Rightarrow (i)$, we show that the quasi-isometries
$$i : \left\{ \begin{array}{ccc} \mathcal{P}_X & \to & X \\ P & \mapsto & x_P \end{array} \right. \text{ and } p : \left\{ \begin{array}{ccc} Y & \to & \mathcal{P}_Y \\ y & \mapsto & \text{piece containing $y$} \end{array} \right.$$
are respectively quasi-$(1/m)$-to-one and quasi-$n$-to-one. It follows from Proposition~\ref{prop:EasyKappa}(ii) that $p \circ f \circ i$ is quasi-one-to-one, so, according to Proposition~\ref{prop:QIdistBij}, it is at finite distance, say $L$, from a bijection $\psi : \mathcal{P}_X \to \mathcal{P}_Y$. Define a map $g : X \to Y$ by
$$g : x \mapsto \text{$y_{\psi(P)}$ where $P \in \mathcal{P}_X$ contains $x$}.$$
Fix a point $x \in X$, and let $P \in \mathcal{P}_X$ denote the piece containing it. Given constants $C,K \geq 0$ such that $f$ is a $(C,K)$-quasi-isometry, we have
$$\begin{array}{lcl} d(f(x),g(x)) & \leq & d(f(x),f(x_P)) + d(f(x_P),g(x)) \leq C d(x,x_P) +K+ d(f(x_P),g(x)) \\ \\ & \leq & CD+K+ d(f(x_P),g(x)). \end{array}$$
Notice that $f(x_P)$ belongs to $p \circ f \circ i(P)$ and $g(x)$ to $\psi(P)$. Consequently, these two pieces must be within $L$ in $\mathcal{P}_Y$, which implies that $d(f(x_P),g(x)) \leq (L+1)D$. Thus, we have proved that $f$ and $g$ are at distance at most $CD+K+(L+1)D$, concluding the proof. 
\end{proof}

\section{Examples of measure-scaling quasi-isometries}

Our main source of examples of standard metric measure spaces are compactly generated locally compact groups. Clearly the notion of measure-scaling quasi-isometry does not depend on a specific choice of word metric. Indeed, if $S_1$ and $S_2$ are two generating compact subsets, the identity map $(G,d_{S_1},\mu)\to(G,d_{S_2},\mu)$ is a bijective quasi-isometry (actually even bi-Lipschitz). In particular it is quasi-one-to-one.
We are going to prove a much more general statement that can be seen as a refinement of the fundamental theorem of geometric group theory. 

\begin{thm}\label{thm:FundThm}
Let $(X,d,\mu)$ be a standard metric measure space. Let $G$ be a locally compact group acting continuously by measure-preserving isometries on $X$. We assume that the action is proper and cocompact. Then $G$ is compactly generated and the orbit maps are measure-scaling quasi-isometries. 
\end{thm}
\begin{remark}
As observed in Remarks \ref{rem:scalingNonAmen} and \ref{rem:geomAmenable} the previous theorem is only meaningful when the group $G$ is amenable and unimodular.
\end{remark}

\begin{proof}
We know from the fundamental theorem of geometric group theory that $G$ is compactly generated and that the orbit maps are quasi-isometries (see \cite[Theorem~4.C.5. and Corollary~4.C.6]{CornHarpe} for instance). We are going to show that they are measure-scaling quasi-isometries, and compute their scaling factor. In what follows, we refer to \cite{ActionsLocalCompact} for basic material about actions of locally compact groups. We fix a Haar measure $\mu$ on $G$. The action being proper, it is smooth, and so admits a fundamental domain $D$: i.e.\ a set that meets every orbit exactly once. Moreover each ergodic component is reduced to a single orbit $G\cdot x$. Hence it is isomorphic to $(G/G_x,\mu_x)$ where  $G_x$, the stabilizer of $x$ is a compact subrgoup of $G$, and where $\mu_x$ is the image of the Haar measure under the projection $G\to G/G_x$.  Since the action is moreover cocompact, $D$ can be taken to be relatively compact (hence bounded as $d$ is proper). Combining these facts, we deduce that there exists a finite measure $\nu$ on $D$ such that for all measurable function $f$ on $X$
\begin{equation}\label{eq:desint}
\int_X f d\sigma= \int_{D}\left(\int_{G/G_x}f(y)d\mu_x\right)d\nu(x). 
\end{equation}
Note that since the action is by isometries, all orbits maps are at bounded distance from one-another: more precisely, for all $x,x'\in D$, and all $g\in G$, $d(gxgx')=d(x,x')\leq \mathrm{Diam}(D)$. We fix $x_0\in D$, and consider the orbit map $\phi(g)=gx_0$. A quasi-inverse $\psi$ can be defined as follows: for all $x\in D$, and all $y\in G\cdot x$, pick $g$ such that $y=gx$. Indeed, we have that $\phi\circ \psi(y)=gx_0$, where $g$ is such that $gx=y$ for some $x\in D$. We deduce that $d(y,\phi\circ \psi(y))=d(gx,gx_0)=d(x,x_0)\leq  \mathrm{Diam}(D)$.
On the other hand $\psi\circ\phi(g)=\psi(gx_0)=g'$ such that $g'x_0=g x_0$. Hence $g'\in gG_{x_0}$. Hence $d(\psi\circ\phi(g),g)=d(g,g')\leq  \mathrm{Diam}(G_{x_0})<\infty $ since by assumption, $G_{x_0}$ is compact.

Note that by definition of the measure on $G/G_x$, each orbit map $G\to G/G_x$ is measure one-to-one. Hence it follows from the disintegration formula (\ref{eq:desint}) that $\psi$ is measure $\nu(D)$-to-one. In particular it is quasi-$\nu(D)$-to-one. By Proposition \ref{prop:EasyKappa}, we deduce that the orbit maps are quasi-$\nu(D)^{-1}$-to-one, proving the theorem.
\end{proof}
We now list a few corollaries of interest.
\begin{cor}
Let $X$ be locally finite graph and $G$ a closed cocompact subgroup of isometries of $X$ equipped with a left Haar measure $\mu$. Then every orbit map is a measure-scaling quasi-isometry, of scaling factor $\kappa=\sum_{x\in D}\mu(G_x)^{-1}$, where $D$ is a (finite) set of representatives of each orbit, and $G_x$ is the stabilizer of $x$ in $G$. 
\end{cor}
\begin{proof}
The disintegration formula (\ref{eq:desint})  takes the following discrete form in this case:
\[\sum_{v\in X} f(v)= \sum_{x\in D}\mu(G_x)^{-1}\sum_{y\in G/Gx}f(y)\]
In other words $\nu(x)=\mu(G_x)^{-1}$, so $\nu(D)=\sum_{x\in D}\mu(G_x)^{-1}$.
\end{proof}
\begin{cor}
Assume that $G$ is  the fundamental group of a compact Riemannian manifold $M$. Then the inclusion of $G$ in the universal cover $\tilde{M}$ of $M$, equipped with its Riemannian measure is a measure-scaling quasi-isometry of scaling factor the inverse of the volume of $M$. 
\end{cor}
\begin{proof}
This follows from the fact that the volume of $M$ equals the measure of any fundamental domain $D$ for the action of $G$ on $\tilde{M}$.
\end{proof}
\begin{cor}\label{cor:copsi}
Let $G$ and $H$ be compactly generated locally compact groups equipped with left Haar measures $\mu_G$ and $\mu_H$ and word metrics associated to compact generating subsets. A continuous proper cocompact morphism $\phi:G\to H$ is a measure-scaling quasi-isometry. A scaling factor is given by the positive number $\kappa>0$ such that $\phi\ast \mu_G=\kappa \mu_H$.
\end{cor}
\begin{proof}
The group $H$ equipped with a left-invariant word metric associated to a compact generating subset and a left Haar measures, and the $G$-action defined by $g\cdot h=\phi(g)h$ satisfy the conditions of the theorem. 
\end{proof}

Recall that commability is the equivalence relation among locally compact groups generated by the existence of a proper continuous morphism with cocompact image \cite{Cornulier_comma}. We immediately deduce the following from Corollary \ref{cor:copsi}.
\begin{cor}\label{cor:commable}
If two compactly generated locally compact groups are commable, then they are measure-scaling quasi-isometric.
\end{cor}
\begin{cor}\label{cor:LatticesScaing}
Assume $\Gamma_1$ and $\Gamma_2$ are two finitely generated uniform lattices in a locally compact group $G$. Then there is a measure-scaling quasi-isomertry $\Gamma_1\to \Gamma_2$ with scaling factor $\kappa=\mu(G/\Gamma_2)/\mu(G/\Gamma_1)$.
\end{cor}
\begin{proof}
Having a finitely generated uniform lattice, $G$ itself is compactly generated and the inclusion maps $\phi_i:\Gamma_i\to G$, $i=1,2$ are quasi-$\mu(G/\Gamma_i)^{-1}$-to-one by Corollary \ref{cor:copsi}. Letting $\bar{\phi_2}$ be a quasi-inverse of $\phi_2$, the corollary follows from Proposition \ref{prop:EasyKappa}.
\end{proof}

Let us mention the following interest special case.
\begin{cor}
Assume $\Gamma_1$ and $\Gamma_2$ are finite index finitely generated subgroups of a group $\Gamma$. Then there is a measure-scaling quasi-isometry $\Gamma_1\to \Gamma_2$ with scaling factor $\kappa=[\Gamma:\Gamma_2]/[\Gamma:\Gamma_1]$.
\end{cor}
\section{The scaling group of an amenable space}

Given a metric space $X$, we recall that 
\[\mathrm{QI}(X):= \{ \text{quasi-isometry } X \to X \} / \text{bounded distance}\]  forms a group called the \emph{quasi-isometry group} of $X$. Its isomorphism class is clearly invariant under quasi-isometry. 

\begin{definition}
Let $X$ be a standard metric measure space. We define the \emph{(measure-)scaling quasi-isometry group} as follows
\[\mathrm{QI_{sc}}(X):= \{ \text{measure-scaling quasi-isometry } X \to X \} / \text{bounded distance}.\]
\end{definition}

\begin{remark}\label{rem:commability}
By Proposition \ref{prop:EasyKappa}, $\mathrm{QI_{sc}}(X)$ is a subgroup of $\mathrm{QI}(X)$ that is invariant under measure-scaling quasi-isometries (a priori not under quasi-isometries). It follows from Corollary \ref{cor:copsi} that $\mathrm{QI_{sc}}(X)$ is a commability invariant. 
Recall that commability is suitable generalization of commensurability in the context of locally compact groups coined by Yves Cornulier: it is the equivalence relation between locally compact groups generated by the existence of a continuous proper cocompact morphism $G\to H$.
\end{remark}

\begin{remark}
Note that when $X$ is non-amenable, $\mathrm{QI_{sc}}(X)=\mathrm{QI}(X)$. 
\end{remark}
\begin{definition}
If $X$ is amenable, then by Lemma \ref{lem:KappaWellDefined} and  Proposition \ref{prop:EasyKappa}, we have that the scale of a measure-scaling quasi-isometry of $X$ defines a morphism \[\mathrm{scale}:\mathrm{QI_{sc}}(X)\to \mathbb{R}_{>0},\] that we shall call the \emph{scaling morphism}.
The image of $\mathrm{scale}$ will be denoted by $\mathrm{Sc(X)}$ and called the \emph{(measure-)scaling group} of $X$. 
\end{definition}
\begin{remark}
Once again by Proposition \ref{prop:EasyKappa}, $\mathrm{Sc(X)}$ is a measure-scaling quasi-isometry invariant of the space $X$.
Like $\mathrm{QI}(X)$, $\mathrm{QI_{sc}}(X)$ is in general a complicated object, so it might be convenient to focus on the much simpler invariant $\mathrm{Sc(X)}$. 
\end{remark}

We now list a few examples of families of amenable groups $G$ for which $\mathrm{Sc}(G)$ can be computed.

\begin{cor}\label{cor:scalinggroupR}
$\mathrm{Sc}(G)=\mathbb{R}_{>0}$ in the following cases:
\begin{itemize}
\item[(i)] Carnot groups and their lattices;
\item[(ii)] $G=\mathrm{SOL}(\mathbb{R})$ or a lattice in $G$;
\item[(iii)] $G$ is a solvable Baumslag-Solitar group $\mathrm{BS}(1,n)$ for some $n\geq 2$.
\end{itemize}
Moreover in the case of (ii) and (iii), every quasi-isometry is a measure-scaling quasi-isometry, i.e.\ $\mathrm{QI}(X)=\mathrm{QI_{sc}}(X)$.
\end{cor}
\begin{proof}
 Recall that a Carnot group is a nilpotent simply connected Lie group with a one-parameter group of automorphism that act by homotheties on the abelianization. Clearly for such a group we have $\mathrm{Sc_{Aut}}(G)=\mathbb{R}_{>0}$. 
Recall that $\mathrm{SOL}(\mathbb{R})=\mathbb{R}^2\rtimes_{(t,t^{-1})}\mathbb{R}_{>0}$ is a normal subgroup in 
the direct product of two copies of the affine group $\mathrm{Aff}(\mathbb{R})=\mathbb{R}\rtimes_{t}\mathbb{R}_{>0}$, where each copy of $\mathbb{R}_{>0}$ acts by conjugation on $\mathrm{SOL}(\mathbb{R})$ in a non-meseasure-preserving way. We therefore deduce that $\mathrm{Sc_{Aut}}(G)=\mathbb{R}_{>0}$. 

\noindent The case of $\mathrm{BS}(1,n)$ is similar: $\mathrm{BS}(1,n)$ is a cocompact lattice in the group $(\mathbb{R}\times\mathbb{Q}_n)\rtimes_{(n^k,n^k)}\mathbb{Z}$, which is a normal subgroup in the direct product of $\mathrm{Aff}(\mathbb{R})$ and $\mathbb{Q}_n\rtimes_{n^k} \mathbb{Z}$. The $\mathbb{R}_{>0}$ subgroup of $\mathrm{Aff}(\mathbb{R})$  acts by conjugation on $(\mathbb{R}\times\mathbb{Q}_n)\rtimes_{(n^k,n^k)}\mathbb{Z}$ in non-measure-preserving way, so once again $\mathrm{Sc_{Aut}}(G)=\mathbb{R}_{>0}$. 
The last statement follows from the description of quasi-isometries given respectively in \cite{MR1608595} for Baumslag-Solitar groups and in \cite{EFWII} for $G=\mathrm{SOL}(\mathbb{R})$.
\end{proof}

\begin{remark}
The first class of examples covers a large class of nilpotent groups including abelian groups, the Heisenberg group, and more generally free nilpotent groups of finite step and rank. The example of SOL is very particular, and it is easy to see that any unimodular group of the form $\mathbb{R}^k\rtimes \mathbb{R}^l$ satisfy $\mathrm{Sc_{Aut}}(G)=\mathbb{R}_{>0}$ (we leave the proof to the reader). 
 \end{remark}

It is interesting to note that the scaling group can also be smaller in a number of situations. However the list of such examples is constrained by the fact that proving any restriction on the scaling group requires a deep understanding of the quasi-isometries of the space, an information that is not yet available for a very wide class of amenable groups.
\begin{prop}\label{prop:Sc(G)}
Let $G$  be an amenable finitely presented group, and let $F$ be a finite group of size $n\geq 2$.
\begin{itemize}
\item[(i)] If $G$ is one-ended, then $\mathrm{Sc}(F\wr G)=\{1\}$;
\item[(ii)] if $G$ is two-ended, then \[\mathrm{Sc}(F \wr G)=\{p_1^{n_1}p_2^{n_2}\ldots p_k^{n_k}\mid n_1,\ldots,n_k\in \mathbb{Z}\}\]
where $p_1,\ldots,p_k$ are the primes appearing in the prime decomposition of $n$. 
\end{itemize}
\end{prop}
\begin{proof}
(i) is proved in \cite{QIwreath}, while (ii) is a consequence of the description the quasi-isometry group of $F\wr \mathbb{Z}$ given in \cite{EFWII} (see \cite{MR2730576}). 
\end{proof}

\section{Questions}
We start with a natural question on measure-scaling quasi-isometries.
 \begin{question}
Are there amenable groups that are quasi-isometric but not measure-scaling quasi-isometric?
 \end{question}
Recall that commable groups are measure-scaling quasi-isometric (see Corollary \ref{cor:commable}). Hence finding such a pair would in particular mean finding a pair of amenable groups that are quasi-isometric but not commable. As far as we know this is  unknown. 

A way to distinguish groups up to measure-scaling quasi-isometries is to show that they have distinct scaling groups. This yields the following refinement of the previous question.
 \begin{question}
Are there pairs of amenable groups that are quasi-isometric but have distinct scaling groups?
 \end{question}

Recall that $\mathrm{Sc}(G)=\mathbb{R}_{>0}$ for all Carnot group and more generally for all simply connected nilpotent groups admitting a one parameter group of automorphisms that do not preserve the Haar measure of $G$.  This motivates the following question.
 \begin{question}
 Do we have $\mathrm{Sc}(G)=\mathbb{R}_{>0}$ for all locally compact group $G$ with polynomial growth? 
We might be even more ambitious and ask whether this holds for all amenable connected Lie groups?
 \end{question}
 Since a polycyclic group is virtually a uniform lattice in a solvable connected Lie group, a positive answer to the previous question would yield the same conclusion for polycyclic groups.
 
Alternatively  we might expect the following opposite behavior.
 \begin{question}
 Do we have $\mathrm{Sc}(G)=\mathrm{Sc_{Aut}}(G)$ for all simply connected nilpotent Lie group $G$? For all simply connected unimodular solvable Lie groups? For $p$-adic solvable algebraic groups?
 \end{question}
 
Recall that the automorphism group of a simply connected Lie group coincides with the automorphism group of its Lie algebra. Hence it is a real algebraic group. Since the scaling morphism is obviously Zariski  continuous, it follows that $\mathrm{Sc_{Aut}}(G)$ is either reduced to $1$ or all of $\mathbb{R}_{>0}$. 
A similar observation holds for $p$-adic algebraic groups.
Since their automorphism group is also an algebraic $p$-adic group, $\mathrm{Sc_{Aut}}(G)$ is either trivial or generated by a power of $p$. 
Consider for instance the following example, which should be accessible using methods from \cite{EFWII}.
\begin{conj}
The scaling group of a unimodular solvable algebraic group over  $\mathbb{Q}_n$ is a subgroup (possibly trivial) of $\{p_1^{n_1}p_2^{n_2}\ldots p_k^{n^k}\mid n_1,\ldots,n_k\in \mathbb{Z}\}$, where $p_1,\ldots,p_k$ are the primes appearing in the prime decomposition of $n$. 
\end{conj}
Note that the scaling group of  $\mathrm{SOL}(\mathbb{Q}_n)$ equals $\{p_1^{n_1}p_2^{n_2}\ldots p_k^{n^k}\mid n_1,\ldots,n_k\in \mathbb{Z}\}$. Indeed, this group acts geometrically on the Diestel-Leader graph $\mathrm{DL}(n,n)$, and therefore is commable to the lamplighter $\mathbb{Z}/n \mathbb{Z}\wr \mathbb{Z}$. So the conclusion follows from  Proposition~\ref{prop:Sc(G)}.

 We end with a wider problem concerning the scaling group.
\begin{question}
What subgroups of $\mathbb{R}_{>0}$ can be realized as scaling groups of finitely generated amenable groups? For instance can it be not contained in $\mathbb{Q}$ and yet not be all of $\mathbb{R}_{>0}$?
\end{question}

\addcontentsline{toc}{section}{References}

\bibliographystyle{alpha}
{\footnotesize\bibliography{ScalingQI}}

\Address

\end{document}